\documentclass[12pt,a4paper,reqno]{amsproc}
\usepackage{amsmath}
\usepackage{amssymb}
\usepackage{amsthm}
\usepackage[backrefs]{amsrefs}
\usepackage[german,danish,english]{babel}
\usepackage[utf8]{inputenc}
\usepackage[T1]{fontenc}
\usepackage{graphicx}
\usepackage{color}
\usepackage{mathrsfs}
\usepackage{comment}
\usepackage{multirow}
\usepackage{enumitem}
\usepackage[footnote,draft,danish,silent]{fixme}
\usepackage{MnSymbol}
\usepackage{hyphenat}

\numberwithin{equation}{section}

\newcommand{\NN}{\mathbb N}
\newcommand{\ZZ}{\mathbb Z}

\newcommand{\FF}{\mathbb F}
\newcommand{\CC}{\mathbb C}

\DeclareMathOperator{\rank}{rank}
\DeclareMathOperator{\Tr}{Tr}

\DeclareMathOperator{\Hom}{Hom}

\DeclareMathOperator{\Hamming}{Hamming}
\DeclareMathOperator{\HS}{HS}
\DeclareMathOperator{\U}{U}
\DeclareMathOperator{\F}{F}

\DeclareMathOperator{\T}{(T)}
\DeclareMathOperator{\TFD}{(T;FD)}
\DeclareMathOperator{\ptau}{(\tau)}
\DeclareMathOperator{\SL}{SL}
\DeclareMathOperator{\Sp}{Sp}
\DeclareMathOperator{\spn}{span}
\DeclareMathOperator{\End}{End}

\DeclareMathOperator{\Aut}{Aut}
\DeclareMathOperator{\Out}{Out}
\DeclareMathOperator{\MCG}{MCG}

\DeclareMathOperator{\Sym}{Sym}
\DeclareMathOperator{\Stab}{Stab}

\DeclareMathOperator{\Ker}{Ker}

\DeclareMathOperator{\GL}{GL}

\DeclareMathOperator{\BS}{BS}

\DeclareMathOperator{\flex}{flex}
\DeclareMathOperator{\veryflex}{very-flex}
\DeclareMathOperator{\St}{St}

\DeclareMathOperator{\chara}{char}
\DeclareMathOperator{\conj}{Conj}

\newcommand{\frakp}{\mathfrak p}

\newcommand{\calG}{\mathcal G}

\newcommand{\calB}{\mathcal B}

\newcommand{\calP}{\mathcal P}
\newcommand{\calH}{\mathcal H}
\newcommand{\calO}{\mathcal O}

\newtheorem{stn}{Sætning}[section]

\newtheorem{cor}[stn]{Corollary}

\newtheorem{prop}[stn]{Proposition}

\theoremstyle{definition}
\theoremstyle{remark}\newtheorem{rem}[stn]{Remark}

\theoremstyle{definition}\newtheorem{defn}[stn]{Definition}
\newtheorem{lem}[stn]{Lemma}
\newtheorem{thm}[stn]{Theorem}

\title[Group stability and Property (T)]{Group stability and Property (T)}

\author[O.\ Becker]{Oren Becker}
\address{O.B., Einstein Institute of Mathematics, Hebrew University}
\email{oren.becker@mail.huji.ac.il}

\author[A.\ Lubotzky]{Alexander Lubotzky}
\address{A.L., Einstein Institute of Mathematics, Hebrew University}
\email{alex.lubotzky@mail.huji.ac.il}

\begin{document}
\begin{abstract}
In recent years, there has been a considerable amount of interest
in the stability of a finitely-generated group $\Gamma$ with respect
to a sequence of groups $\left\{G_{n}\right\}_{n=1}^{\infty}$, equipped with
bi-invariant metrics $\left\{d_{n}\right\}_{n=1}^{\infty}$.
We consider the case $G_{n}=\U\left(n\right)$ (resp.
$G_{n}=\Sym\left(n\right)$), equipped with the normalized Hilbert-Schmidt
metric $d_{n}^{\HS}$ (resp. the normalized Hamming metric $d_{n}^{\Hamming}$).
Our main result is that if $\Gamma$ is infinite, hyperlinear (resp. sofic) and has Property $\T$, then it is not stable with respect to
$\left(\U\left(n\right),d_{n}^{\HS}\right)$
(resp. $\left(\Sym\left(n\right),d_{n}^{\Hamming}\right)$).

This answers a question of Hadwin and Shulman regarding the 
stability of $\SL_{3}\left(\ZZ\right)$.
We also deduce that the mapping class group $\MCG\left(g\right)$, $g\geq 3$, and $\Aut\left(\FF_n\right)$, $n\geq 3$, are not stable with respect to $\left(\Sym\left(n\right),d_{n}^{\Hamming}\right)$.

Our main result exhibits a difference between stability
with respect to the normalized Hilbert-Schmidt metric on
$\U\left(n\right)$ and the (unnormalized) $p$-Schatten metrics,
since many groups with Property $\T$ are stable with respect to the
latter metrics, as shown by De Chiffre-Glebsky-Lubotzky-Thom and
Lubotzky-Oppenheim.

We suggest a more flexible notion of stability that may repair this
deficiency of stability with respect to
$\left(\U\left(n\right),d_{n}^{\HS}\right)$
and $\left(\Sym\left(n\right),d_{n}^{\Hamming}\right)$.
\end{abstract}

\maketitle
\section{Introduction}

Let $\calG=\left\{ \left(G_{n},d_{n}\right)\right\} _{n=1}^{\infty}$
be a family of groups $G_{n}$ endowed with bi-invariant metrics $d_{n}$,
i.e., $d_{n}\left(ag_{1}b,ag_{2}b\right)=d_{n}\left(g_{1},g_{2}\right)$
for all $g_{1},g_{2},a,b\in G_{n}$. Here are some examples:
\begin{enumerate}
\item $\calP=\left\{ \left(\Sym\left(n\right),d_{n}^{\Hamming}\right)\right\} _{n=1}^{\infty}$,
where $d_{n}^{\Hamming}$ is the normalized Hamming metric on $\Sym\left(n\right)$:
for $\sigma,\tau\in\Sym\left(n\right)$,
\[
d_{n}^{\Hamming}\left(\sigma,\tau\right)=\frac{1}{n}\cdot\left|\left\{ x\in\left[n\right]\mid\sigma\left(x\right)\neq\tau\left(x\right)\right\} \right|\,\,,
\]
where $\left[n\right]=\left\{ 1,\dotsc,n\right\} $.
\item $\HS=\left\{ \left(\U\left(n\right),d_{n}^{\HS}\right)\right\} _{n=1}^{\infty}$,
where $d_{n}^{\HS}$ is the normalized Hilbert-Schmidt metric: for
$A,B\in\U\left(n\right)$,
\[
d_{n}^{\HS}\left(A,B\right)=\|A-B\|_{\HS}\,\,\text{,}\,\,\,\,\,\,\,\,\,\,\,\,\text{where }\|T\|_{\HS}=\left(\Tr\left(\frac{1}{n}\cdot T^{*}T\right)\right)^{1/2}\,\,\text{.}
\]
\item $\calG^{\left(p\right)}=\left\{ \left(\U\left(n\right),d_{n}^{\left(p\right)}\right)\right\} _{n=1}^{\infty}$,
for any fixed $1\leq p<\infty$, where $d_{n}^{\left(p\right)}$is
the Schatten $p$-norm (see \cite{bhatia}, Section IV.2): for $A,B\in\U\left(n\right)$,
\[
d_{n}^{\left(p\right)}\left(A,B\right)=\|A-B\|_{p}\,\,\text{,}\,\,\,\,\,\,\,\,\,\,\,\,\text{where }\|T\|_{p}=\left(\Tr\left(\left(T^{*}T\right)^{p/2}\right)\right)^{1/p}\,\,\text{.}
\]
The case $p=2$ is of special interest: this is the standard $L^{2}$-norm,
a.k.a. the Frobenius norm, denoted $\|\cdot\|_{\F}$. Note that $d_{n}^{\left(2\right)}=\sqrt{n}\cdot d_{n}^{\HS}$.
The proofs in Section \ref{sec:proofs} also make use of the Frobenius
norm $\|\cdot\|_{\F}$ for non-square matrices, which is defined by
the same formula: $\|T\|_{\F}=\Tr\left(T^{*}T\right)^{1/2}$.
\item $\calG^{\left(\infty\right)}=\left\{ \left(\U\left(n\right),d_{n}^{\left(\infty\right)}\right)\right\} _{n=1}^{\infty}$,
where the metric $d_{n}^{\left(\infty\right)}$ is defined for $A,B\in\U\left(n\right)$
by $d_{n}^{\left(\infty\right)}\left(A,B\right)=\|A-B\|_{\infty}$,
where $\|\cdot\|_{\infty}$ is the operator norm.
\end{enumerate}
Let $\FF$ be a free group on a finite set $S$. Let $\Gamma$ be
a quotient of $\FF$, and denote the quotient map by $\pi:\FF\rightarrow\Gamma$.
The objects $S$, $\FF$, $\Gamma$ and $\pi$ shall remain fixed throughout this paper.

From now on, for a group $G$, a function $f:S\rightarrow G$ and
an element $w\in\FF$, we write $f\left(w\right)$ for the element
of $G$ resulting from applying the substitution $s\mapsto f\left(s\right)$
to the word $w$.
\begin{defn}
~\renewcommand{\labelenumi}{\roman{enumi})}
\begin{enumerate}
\item A \emph{$\calG$-stability-challenge} for $\Gamma$ is a sequence
$\left(f_{k}\right)_{k=1}^{\infty}$ of functions $f_{k}:S\rightarrow G_{n_{k}}$,
$n_k\in\NN$, such that for every $w\in\Ker\left(\pi\right)$,
\[
d_{n_{k}}\left(f_{k}\left(w\right),{\bf 1}_{G_{n_{k}}}\right)\overset{k\rightarrow\infty}{\longrightarrow}0\,\,\text{.}
\]
\item Let $\left(f_{k}\right)_{k=1}^{\infty}$ be a $\calG$-stability-challenge
for $\Gamma$. A \emph{solution} for $\left(f_{k}\right)_{k=1}^{\infty}$
is a sequence of functions $\left(g_{k}\right)_{k=1}^{\infty}$,
$g_{k}:S\rightarrow G_{n_{k}}$, such that for every $w\in\Ker\left(\pi\right)$,
$g_{k}\left(w\right)={\bf 1}_{G_{n_{k}}}$ (i.e., $g_{k}$ defines a homomorphism $\Gamma\rightarrow G_{n_k}$),
and
\[
\sum_{s\in S}d_{n_{k}}\left(f_{k}\left(s\right),g_{k}\left(s\right)\right)\overset{k\rightarrow\infty}{\longrightarrow}0\,\,\text{.}
\]
\item The group $\Gamma$ is \emph{$\calG$-stable} if every $\calG$-stability-challenge
for $\Gamma$ has a solution.
\end{enumerate}
\end{defn}
While the above definition of a $\calG$-stable group made use of
a given presentation of $\Gamma$ as a quotient of a free group, it
is in fact a group property, independent of the specific presentation
(cf. \cite{arzhantseva-paunescu}).

In recent years, there has been a considerable amount of interest
in ``group stability'' (see \cite{arzhantseva-paunescu}, \cite{thom-icm},
\cite{blt} and the references therein).
One of the main motivations is the study of $\calG$-approximations
of $\Gamma$:
\begin{defn}
For $\calG$ and $\Gamma$ as above, we say that $\Gamma$ is \emph{$\calG$-approximated}
if there is a sequence of integers
$\{n_{k}\}_{k=1}^{\infty}$,
%\$n_k\overset{k\rightarrow\infty}{\longrightarrow}\infty$,
and a sequence
$\left(\varphi_{k}\right)_{k=1}^{\infty}$
of functions $\varphi_{k}:\Gamma\rightarrow G_{n_{k}}$, such that
\[
\forall g,h\in\Gamma\,\,\,\,\,\,\,\,\lim_{k\rightarrow\infty}d_{n_{k}}\left(\varphi_{k}\left(gh\right),\varphi_{k}\left(g\right)\varphi_{k}\left(h\right)\right)=0
\]
and
\[
\forall 1_{\Gamma}\neq g\in\Gamma\,\,\,\,\,\,\,\,
\limsup_{k} d_{n_{k}}\left(\varphi_{k}\left(g\right),1_{G_{n_{k}}}\right)>0
\]
\end{defn}

In classical terminology, $\calP$-approximated groups (for $\calP$
as in Example 1 above) are called \emph{sofic} groups, and $\HS$-approximated
groups are called \emph{hyperlinear} groups
(or \emph{Connes embeddable} \cite{thom-icm}).
It is a well-known open problem,
due to Gromov (resp. Connes), whether every group is sofic (resp. hyperlinear).
Note that all sofic groups are hyperlinear.

In \cite{dglt}, it was shown for the first time that there are finitely
presented groups $\Gamma$ which are not $\left(\U\left(n\right),d_{n}^{\left(2\right)}\right)$-approximated
(i.e., \emph{Frobenius-approximated}), and this result was extended
in \cite{lubotzky-oppenheim} to all $1<p<\infty$. The groups $\Gamma$
in those papers are finite central extensions of suitable lattices
$\bar{\Gamma}$ in simple Lie groups of rank $r\geq3$ over local
non-archimedean fields. The key point there is that these groups $\Gamma$
and $\bar{\Gamma}$ are $\calG^{\left(2\right)}$-stable (and even $\calG^{\left(p\right)}$-stable). This is
proved as a corollary to the vanishing result $H^{i}\left(\Gamma,V\right)=0$
for every $i=1,\dotsc,r-1$ and for all actions of $\Gamma$ on Hilbert
spaces $V$ (and the same for many Banach spaces). The case $i=2$
gives the stability of $\Gamma$. Vanishing for $i=1$
is equivalent to $\Gamma$ having Property $\T$, so all the groups
treated there have Kazhdan's Property (T).

The goal of the present paper is to show that these examples are neither
$\calP$-stable nor $\HS$-stable. In fact, we prove a much more general
result, which is of independent interest:
\begin{thm}
\label{thm:intro-sofic-hyperlinear}~\renewcommand{\labelenumi}{\roman{enumi})}
\begin{enumerate}
\item If $\Gamma$ is sofic and has Property $\T$, then it is not $\calP$-stable,
unless it is finite.
\item If $\Gamma$ is hyperlinear and has Property $\T$, then it is not
$\HS$-stable, unless it is finite.
\end{enumerate}
\end{thm}
As every finitely generated linear group is sofic, all lattices
in higher rank simple algebraic groups over local fields are neither $\calP$-stable nor $\HS$-stable.
In particular, this holds for $\SL_n\left(\ZZ\right)$, $n\geq 3$,
answering a question of Hadwin and Shulman \cite{HS2}.

Theorem \ref{thm:intro-sofic-hyperlinear} is a corollary of the following:
\begin{thm}
\label{thm:intro-stable-fi-subgroups}Assume that $\Gamma$ has Kazhdan's
Property $\T$, and is either $\calP$-stable or $\HS$-stable. Then,
$\Gamma$ has only finitely many finite-index subgroups.
\end{thm}

Theorem \ref{thm:intro-stable-fi-subgroups} is proved in Section
\ref{sec:proofs}. We can already show how it implies Theorem \ref{thm:intro-sofic-hyperlinear}:
\begin{proof}
[Proof of Theorem \ref{thm:intro-sofic-hyperlinear}]Assume that $\Gamma$
is sofic and $\calP$-stable. A well-known observation (see Theorem 2
in \cite{glebsky-rivera}) says that, in this case, $\Gamma$ is residually-finite.
If further, $\Gamma$ has Property $\T$, then by Theorem \ref{thm:intro-stable-fi-subgroups},
it has only finitely many finite-index subgroups, and so it is finite.

Assume, instead, that $\Gamma$ is hyperlinear and $\HS$-stable.
It is easy to see that in this case too, $\Gamma$ is residually-finite.
Indeed, arguing as in \cite{glebsky-rivera}, we see that $\Gamma$
is residually-linear, and so it is residually-finite since finitely-generated
linear groups are residually-finite. So, as before, if, further, $\Gamma$ has
Property $\T$, it must be finite.
\end{proof}

In Section \ref{sec:AutFn-MCG}, we deduce the following:
\begin{thm}
	\label{thm:intro-aut-fn-mcg}For $g\geq 3$, the mapping class group
	$\MCG\left(g\right)$ of an orientable closed surface of genus $g$
	is not $\calP$-stable.
	For $n\geq 3$, the (outer) automorphism group $\Aut\left(\FF_{n}\right)$
	(and $\Out\left(\FF_{n}\right)$) is not $\calP$-stable.
\end{thm}
We do not know if the theorem also holds for $\HS$-stability.
Its proof uses Theorem \ref{thm:intro-sofic-hyperlinear} and Proposition \ref{prop:P-stability-and-quotients},
which states that if $\Gamma$ is $\calP$-stable then so is its quotient by a normal subgroup $N$,
provided that $N$ is a finitely-generated group.

In Section \ref{sec:remarks}, we discuss variations
of Property $\T$ (e.g., Property $\ptau$ and relative Property $\T$),
the stability of semidirect products and free products with amalgamation and a flexible notion of stability, 
and suggest problems for further research.

In the appendix we reprove Theorem \ref{thm:intro-stable-fi-subgroups} using the
representation theory of Chevalley groups, under the assumption that $\Gamma$ has an infinite linear quotient.
In fact, we prove a somewhat stronger version of the theorem for such groups $\Gamma$.

\section*{Acknowledgments}

We thank Narutaka Ozawa for suggesting a proof of the non-$\calP$-stability
of $\SL_{3}\left(\ZZ\right)$, which led to the results described
in this paper.
This work has been supported by grants of the ERC, NSF and BSF.
This work is part of the PhD thesis of the first author at the Hebrew University of Jerusalem.

\section{\label{sec:proofs}The proof of Theorem \ref{thm:intro-stable-fi-subgroups}}

Before we begin, we record a simple observation regarding $\calP$-stability.
Fix formal elements $\left\{ v_{i}\right\} _{i=1}^{\infty}$, and
for every $n\in\NN$, let $\calB_{n}=\left(v_{1},\dotsc,v_{n}\right)$
serve as an ordered basis for a complex vector space $\calH_{n}$.
A permutation $\sigma\in\Sym\left(n\right)\cong\Sym\left(\calB_{n}\right)$
extends uniquely to an element of $\U\left(\calH_{n}\right)\cong\U\left(n\right)$,
giving an embedding $\iota:\Sym\left(n\right)\rightarrow\U\left(n\right)$.
For permutations $\sigma,\tau\in\Sym\left(n\right)$,
\[
d_{n}^{\HS}\left(\iota\left(\sigma\right),\iota\left(\tau\right)\right)=\sqrt{2\cdot d_{n}^{\Hamming}\left(\sigma,\tau\right)}\,\,\text{.}
\]
Therefore, for
$\calG_{0}=\left\{\left( \iota\left(\Sym\left(n\right)\right),d_n^{\HS}\right)\right\}_{n=1}^{\infty}$,
the group
$\Gamma$ is $\calP$-stable if and only if it is $\calG_{0}$-stable.
We also refer to $\calG_0$-stability-challenges and $\calP$-stability challenges interchangeably.

For a linear operator $T:\calH\rightarrow\calH$ on a finite-dimensional vector space $\calH$ and
a basis $\calB$ for $\calH$, we write $\left[T\right]_{\calB}$ for the matrix representing $T$ with
respect to $\calB$.
\begin{lem}
\label{lem:restriction-of-domain-and-range}Let $\calH$ be a finite-dimensional
complex Hilbert space with orthonormal ordered basis $\calB=\left(v_{1},\dotsc,v_{n}\right)$.
Let $\calB_{0}=\left(v_{1},\dotsc,v_{n-1}\right)$ and $\calH_{0}=\spn_{\CC}\calB_{0}$.
Let $T:\calH\rightarrow\calH$ be a linear operator. Write $P_{0}:\calH\rightarrow\calH_{0}$
for the orthogonal projection onto $\calH_{0}$, and $T_{0}=P_{0}\circ T\mid_{\calH_{0}}:\calH_{0}\rightarrow\calH_{0}$.
Then,\renewcommand{\labelenumi}{\roman{enumi})}
\begin{enumerate}
\item If $T$ permutes $\calB$, then there is a linear operator $A_{0}:\calH_{0}\rightarrow\calH_{0}$,
which permutes $\calB_{0}$, such that
\[
\|T_{0}-A_{0}\|_{\F}\leq1
\]
\item If $T$ is unitary, then there is a unitary linear operator $A_{0}:\calH_{0}\rightarrow\calH_{0}$,
such that
\[
\|T_{0}-A_{0}\|_{\F}\leq1
\]
\end{enumerate}
In both cases, the inclusion map $f:\calH_{0}\rightarrow\calH$ satisfies
\[
\|T^{-1}\circ f\circ A_{0}-f\|_{\F}\leq2
\]

\end{lem}

\begin{proof}
(i) Assume that $T$ permutes $\calB$. Denote $T^{-1}\left(v_{n}\right)=v_{i_{0}}$,
$1\leq i_{0}\leq n$. Define a linear operator $A_{0}:\calH_{0}\rightarrow\calH_{0}$
on the elements of the basis $\calB_{0}$ by 
\[
A_{0}\left(v_{i}\right)=\begin{cases}
T\left(v_{i}\right) & i\neq i_{0}\\
T\left(v_{n}\right)=T\left(T\left(v_{i}\right)\right) & i=i_{0}
\end{cases}\,\,\,\text{.}
\]
Then, only the $i_{0}$-th column of $\left[T_{0}-A_{0}\right]_{\calB_{0}}$
may be nonzero, and its norm is $0$ if $i_{0}=n$, or $1$
otherwise. In any case, 
\[
\|T_{0}-A_{0}\|_{\F}\leq1
\]

(ii) Assume that $T$ is unitary. Take a polar decomposition $T_{0}=A_{0}\cdot\sqrt{T_{0}^{*}T_{0}}$
of $T_{0}$, where $A_{0}\in\U$$\left(\calH_{0}\right)$ (see Theorem
3.1.9(c) in \cite{horn-johnson-1994}). Note that, generally, $A_{0}$
is only guaranteed to exist, but is not unique. Then,
\begin{equation}
\|T_{0}-A_{0}\|_{\F}=\|A_{0}^{-1}\cdot\left(T_{0}-A_{0}\right)\|_{\F}=\|\sqrt{T_{0}^{*}T_{0}}-I\|_{\F}\label{eq:unitary-T0-A0}
\end{equation}
Let $u\in M_{1\times\left(n-1\right)}\left(\CC\right)$ be the bottom
row of $\left[T\right]_{\calB}$, with the rightmost entry removed.
Since $T$ is unitary, we have $\|u\|\leq1$. Partition $\left[T\right]_{\calB}$
as a block matrix, where $\left[T_{0}\right]_{\calB_{0}}$ is the
top-left block, and $u$ is the bottom-left block. Since $T^{*}T=I$,
we get $\left[T_{0}^{*}T_{0}\right]_{\calB_{0}}+u^{*}u=I_{n-1}$,
i.e., $\left[T_{0}^{*}T_{0}\right]_{\calB_{0}}=I_{n-1}-u^{*}u$. The
eigenvalues of $u^{*}u$ are $0$ (with multiplicity $n-2$), and
$\langle u,u\rangle=\|u\|^{2}$ (with multiplicity $1$, corresponding
to the right eigenvector $u^{*}$). So, $\sqrt{T_{0}^{*}T_{0}}$ is
a unitarily diagonalizable operator whose eigenvalues are $1$, with
multiplicity $n-2$, and $\sqrt{1-\|u\|^{2}}$, with multiplicity
$1$. Hence,
\[
\|\sqrt{T_{0}^{*}T_{0}}-I\|_{\F}=\left|\sqrt{1-\|u\|^{2}}-1\right|\le1
\]
which, together with (\ref{eq:unitary-T0-A0}), implies the desired
result.

\smallskip{}
As for the last claim,
\begin{align*}
\|T^{-1}\circ f\circ A_{0}-f\|_{\F} & =\|f\circ A_{0}-T\circ f\|_{\F}\\
 & =\|f\circ A_{0}-T\mid_{\calH_{0}}\|_{\F}\\
 & \leq\|f\circ\left(A_{0}-T_{0}\right)\|_{\F}+\|f\circ T_{0}-T\mid_{\calH_{0}}\|_{\F}\\
 & =\|T_{0}-A_{0}\|_{\F}+\|\left[\left(f\circ P_{0}\circ T-T\right)\mid_{\calH_{0}}\right]_{\calB_{0}}^{\calB}\|_{\F}\\
 & =\|T_{0}-A_{0}\|_{\F}+\|u\|_{\F}\\
 & \leq 1+1=2
\end{align*}
where $u$ is, again, the bottom-left row of $\left[T\right]_{\calB}$,
with the rightmost entry removed.

\end{proof}
\begin{lem}
\label{lem:perturbation-of-unitaries}Let $\calH$ be a finite-dimensional
complex Hilbert space. Let $U_{1},\dotsc,U_{l}\in\U$$\left(\calH\right)$ and 
$E_{1},\dotsc,E_{l}\in\End_{\CC}\left(\calH\right)$. Let $c\geq0$,
and assume that $\|E_{i}\|_{\F}\leq c$ for all $1\leq i\leq l$. Then,
\[
\|\prod_{i=1}^{l}\left(U_{i}+E_{i}\right)-\prod_{i=1}^{l}U_{i}\|_{\F}\leq\left(c+1\right)^{l}
\]
\end{lem}

\begin{proof}
Let $\emptyset\neq A_{0}\subset\left[l\right]$ (where $\left[l\right]=\left\{ 1,\dotsc,l\right\} $).
For each $1\leq i\leq l$, denote $M_{i}=\begin{cases}
E_{i} & i\in A_{0}\\
U_{i} & i\notin A_{0}
\end{cases}$. Let $1\leq k\leq l$, and consider the product $\prod_{i=1}^{k}M_{i}$.
On one hand, if $k\notin A_{0}$, then
\[
\|\prod_{i=1}^{k}M_{i}\|_{\F}=\|\left(\prod_{i=1}^{k-1}M_{i}\right)\cdot U_{k}\|_{\F}=\|\left(\prod_{i=1}^{k-1}M_{i}\right)\|_{\F}
\]
since the Frobenius norm $\|\cdot\|_{\F}$ is invariant under multiplication
by unitary operators. On the other hand, if $k\in A_{0}$, then
\[
\|\prod_{i=1}^{k}M_{i}\|_{\F}=\|\left(\prod_{i=1}^{k-1}M_{i}\right)\cdot E_{k}\|_{\F}\leq\|\prod_{i=1}^{k-1}M_{i}\|_{\F}\cdot\|E_{k}\|_{\F}
\]
since $\|\cdot\|_{\F}$ is submultiplicative. So, we conclude by induction
that
\[
\|\prod_{i=1}^{k}M_{i}\|_{\F}\leq\prod_{i\in A_{0}}\|E_{i}\|_{\F}\leq c^{\left|A_{0}\right|}\,\,\text{.}
\]
Together with the triangle inequality, this implies that

\begin{align*}
\|\prod_{i=1}^{l}\left(U_{i}+E_{i}\right)-\prod_{i=1}^{l}U_{i}\|_{\F} & \leq\sum_{\emptyset\neq A\subset\left[l\right]}c^{\left|A\right|}\\
 & =\sum_{i=1}^{l}\binom{l}{i}\cdot c^{i}\\
 & \leq\left(c+1\right)^{l}\,\,\text{.}
\end{align*}
\end{proof}

For a word $w\in\FF$, write $\left|w\right|$ for the \emph{length} of $w$, i.e., the length of $w$ when written as a reduced
word over $S^{\pm}$.
Recall that we write $\pi$ for the fixed quotient map $\pi:\FF\rightarrow\Gamma$.
\begin{prop}
\label{prop:restriction-is-almost-solution-with-almost-morphism}Let
$\left(\calH,\alpha\right)$ be a finite-dimensional unitary representation
of $\Gamma$. Let $\calH_{0}\subset\calH$ be a subspace of co-dimension
$1$. Let $\calB_{0}\subset\calB$ be orthonormal bases for $\calH_{0},\calH$, respectively.
Then,\renewcommand{\labelenumi}{\roman{enumi})}
\begin{enumerate}
\item There is a function $\rho:S\rightarrow\U\left(\calH_{0}\right)$,
such that $\|\rho\left(w\right)-I\|_{\F}\leq3^{\left|w\right|}$ for
every $w\in\Ker\left(\pi\right)$, and the inclusion map $f:\calH_{0}\rightarrow\calH$
satisfies
\[
\|\alpha\left(s^{-1}\right)\circ f\circ\rho\left(s\right)-f\|_{\F}\leq2
\]
for each $s\in S$.
\item If, furthermore, each $\alpha\left(s\right)$ permutes $\calB$,
then $\rho$ above can be chosen such that each $\rho\left(s\right)$
permutes $\calB_{0}$.
\end{enumerate}
\end{prop}

\begin{proof}
Define $\alpha_{0}:S\rightarrow\End_{\CC}\calH$ by $\alpha_{0}\left(s\right)=P_{0}\circ\alpha\left(s\right)\mid_{\calH_{0}}$,
where $P_{0}:\calH\rightarrow\calH_{0}$ is the orthogonal projection.
By Lemma \ref{lem:restriction-of-domain-and-range}, applied to $\alpha\left(s\right)$
for each $s\in S$ separately, there is a function $\rho:S\rightarrow\U\left(\calH_{0}\right)$,
such that
\begin{equation}
\|\alpha_{0}\left(s\right)-\rho\left(s\right)\|_{\F}\le1\label{eq:applied-unitary-correction}
\end{equation}
and if each $\alpha\left(s\right)$ permutes $\calB$, then $\rho$ can be chosen so that each $\rho\left(s\right)$
permutes $\calB_{0}$. In any case, Lemma \ref{lem:restriction-of-domain-and-range}
guarantees that
\[
\|\alpha\left(s^{-1}\right)\circ f\circ\rho\left(s\right)-f\|_{\F}\leq2
\]
for each $s\in S$. Let $s\in S$.
By considering the matrix representations of
$\rho\left(s\right)\oplus{\bf 1}_{\calH_{0}^{\perp}}$ and
$\alpha\left(s\right)$, and using the fact
that $\alpha\left(s\right)$ is unitary, we see that
\[
\|\rho\left(s\right)\oplus{\bf 1}_{\calH_{0}^{\perp}}-\alpha\left(s\right)\|_{\F}^{2}\leq\|\rho\left(s\right)-\alpha_{0}\left(s\right)\|_{\F}^{2}+3\,\,\text{.}
\]
Hence, from (\ref{eq:applied-unitary-correction}), we conclude that
\begin{align*}
\|\rho\left(s\right)\oplus{\bf 1}_{\calH_{0}^{\perp}}-\alpha\left(s\right)\|_{\F} & \leq\left(\|\rho\left(s\right)-\alpha_{0}\left(s\right)\|_{\F}^{2}+3\right)^{1/2}\leq2
\end{align*}
We would like to bound $\|\rho\left(\cdot\right)\oplus{\bf 1}_{\calH_{0}^{\perp}}-\alpha\left(\cdot\right)\|_{\F}$,
evaluated at a word $w\in\FF$, and so we also need to bound
$\|\rho\left(s\right)^{-1}\oplus{\bf 1}_{\calH_{0}^{\perp}}-\alpha\left(s\right)^{-1}\|_{\F}$.
But, in general $\|A^{-1}-B^{-1}\|_{\F}=\|A-B\|_{\F}$ for $A,B\in\U\left(\calH\right)$,
and so
\[
\|\rho\left(s\right)^{-1}\oplus{\bf 1}_{\calH_{0}^{\perp}}-\alpha\left(s^{-1}\right)\|_{\F}\leq2\,\,\text{.}
\]

Let $w\in\FF$. By Lemma \ref{lem:perturbation-of-unitaries}, the
above implies that
\[
\|\rho\left(w\right)\oplus{\bf 1}_{\calH_{0}^{\perp}}-\alpha\left(w\right)\|_{\F}\leq\left(2+1\right)^{\left|w\right|}=3^{\left|w\right|}\,\,\text{.}
\]
Assume further that $w\in\Ker\left(\pi\right)$. Then $\alpha\left(w\right)=I$,
and so,
\begin{align*}
\|\rho\left(w\right)-I\|_{\F} & =\|\rho\left(w\right)\oplus{\bf 1}_{\calH_{0}^{\perp}}-I\|_{\F}\\
 & \leq\|\rho\left(w\right)\oplus{\bf 1}_{\calH_{0}^{\perp}}-\alpha\left(w\right)\|_{\F}+\|\alpha\left(w\right)-I\|_{\F}\\
 & \leq3^{\left|w\right|}
\end{align*}
\end{proof}

Henceforth, given representations $\calH_1$ and $\calH_2$ of $\Gamma$, we treat
$\Hom_{\CC}\left(\calH_1,\calH_2\right)$ as a $\Gamma$-representation with the action given by
$g\cdot f=g\circ f\circ g^{-1}$ for $g\in\Gamma$ and $f\in\Hom_{\CC}\left(\calH_1,\calH_2\right)$.

\begin{prop}
\label{prop:inclusion-map-is-far-from-morphism}Let $\calH_{0}\subsetneq\calH$
be finite-dimensional complex Hilbert spaces, and write $f:\calH_{0}\rightarrow\calH$
for the inclusion map. Let $\alpha:\Gamma\rightarrow\U\left(\calH\right)$
and $\beta:\Gamma\rightarrow\U\left(\calH_{0}\right)$ be unitary
representations. Then,\renewcommand{\labelenumi}{\roman{enumi})}
\begin{enumerate}
\item If $\left(\calH,\alpha\right)$ is irreducible, then $\|f-h\|_{\F}=\|f\|_{\F}$
for every morphism of $\Gamma$-representations $h:\calH_{0}\rightarrow\calH$.
\item If $\calB_{0}\subset\calB$ are orthonormal bases for $\calH_{0},\calH$, respectively,
each $\beta\left(s\right)$ permutes $\calB_{0}$, each $\alpha\left(s\right)$ permutes $\calB$, and
the action $\Gamma{\curvearrowright^{\alpha}}\calB$ of $\Gamma$ on $\calB$
through $\alpha$ is transitive, then $\|f-h\|_{\F}\geq\frac{1}{\sqrt{2}}\cdot\|f\|_{\F}$
for every morphism of $\Gamma$-representations $h:\calH_{0}\rightarrow\calH$.
\end{enumerate}
\end{prop}

\begin{proof}
(i) Since $\dim_{\CC}\calH_{0}<\dim_{\CC}\calH$ and $\left(\calH,\alpha\right)$
is irreducible, Schur's Lemma implies the only morphism of representations
$\calH_{0}\rightarrow\calH$ is the zero morphism, and so the result
follows.

(ii) For $b_{0}\in\calB_{0}$ and $b\in\calB$, let $E_{b_{0},b}:\calH_{0}\rightarrow\calH$
be the linear map sending $b_{0}\mapsto b$, and sending every other
element of $\calB_{0}$ to zero. Then, $\left\{ E_{b_{0},b}\right\} _{\left(b_{0},b\right)\in\calB_{0}\times\calB}$
is a basis for $\Hom_{\CC}\left(\calH_{0},\calH\right)$. The inner
product for which $\left\{ E_{b_{0},b}\right\} _{\left(b_{0},b\right)\in\calB_{0}\times\calB}$
is an orthonormal basis makes the $\Gamma$-representation $\Hom_{\CC}\left(\calH_{0},\calH\right)$
unitary. The group $\Gamma$ acts on $\calB_{0}\times\calB$ by
$\gamma\cdot\left(b_0,b\right)=\left(\gamma\cdot b_0,\gamma\cdot b\right)$.
A map $T\in\Hom_{\CC}\left(\calH_{0},\calH\right)$,
represented as $T=\sum_{\left(b_{0},b\right)\in\calB_{0}\times\calB}\lambda_{b_{0},b}\cdot E_{b_{0},b}$,
is a morphism of representations if and only if the mapping $\left(b_{0},b\right)\mapsto\lambda_{b_{0},b}$
is constant on each $\Gamma$-orbit of $\calB_{0}\times\calB$.

Let $\calO^{\calB_0}_1,\dotsc,\calO^{\calB_0}_k$ be the orbits of the action
$\Gamma\curvearrowright\calB_0$.
Take an orbit $\calO$ of the action $\Gamma\curvearrowright\calB_0\times\calB$.
Then, there is a unique $1\leq i\leq k$ such that $\calO$ is contained in $\calO^{\calB_0}_i\times \calB$.
We
claim that $\left|\calO\right|\geq 2\cdot\left|\calO^{\calB_0}_i\right|$.
Indeed, let $\left(b_{0},b\right)\in\calO$. Then,
\begin{equation}
\Stab_{\Gamma}\left(\left(b_{0},b\right)\right)=\Stab_{\Gamma}\left(b_{0}\right)\cap\Stab_{\Gamma}\left(b\right)\leq\Stab_{\Gamma}\left(b_{0}\right)\,\,\text{.}\label{eq:inclusion-of-stabs}
\end{equation}
The action $\Gamma{\curvearrowright^{\alpha}}\calB$ is transitive,
and so $\left|\calB\right|=\left[\Gamma:\Stab_{\Gamma}\left(b\right)\right]$.
Thus,
\[
\left[\Gamma:\Stab_{\Gamma}\left(b_{0}\right)\right]\leq\left|\calB_{0}\right|<\left|\calB\right|=\left[\Gamma:\Stab_{\Gamma}\left(b\right)\right]\,\,\text{.}
\]
In particular, $\Stab_{\Gamma}\left(b_{0}\right)$ is not a subgroup
of $\Stab_{\Gamma}\left(b\right)$, and so the inclusion in (\ref{eq:inclusion-of-stabs})
is strict. Hence,
\begin{align}
\left|\calO\right| &=
 \left[\Gamma:\Stab_{\Gamma}\left(\left(b_{0},b\right)\right)\right]\nonumber \\
 & =\left[\Gamma:\Stab_{\Gamma}\left(b_{0}\right)\right]\cdot\left[\Stab_{\Gamma}\left(b_{0}\right):\Stab_{\Gamma}\left(\left(b_{0},b\right)\right)\right]\nonumber \\
 & \geq2\cdot\left[\Gamma:\Stab_{\Gamma}\left(b_{0}\right)\right]\nonumber \\
 & =2\cdot\left|\calO^{\calB_{0}}_i\right|\,\,\text{,}
 \label{eq:orbits-are-large}
\end{align}
as claimed.

For each $\Gamma$-orbit $\calO$ of $\calB_{0}\times\calB$,
let $c\left(\calO\right)$ be the number of elements $\left(b_{0},b\right)\in \calO$
for which $f\left(b_{0}\right)=b$, i.e.,
$c\left(\calO\right)=|\calO\cap\left\{\left(b_{0},b_{0}\right)\mid b_{0}\in\calB_{0}\right\}|$.
Then, $\sum_{\calO}c\left(\calO\right)=\left|\calB_{0}\right|$,
and for each $1\leq i\leq k$,
$\sum_{\calO\subset \calO^{\calB_0}_i\times\calB}c\left(\calO\right)=\left|\calO^{\calB_{0}}_i\right|$.
Let $h:\calH_{0}\rightarrow\calH$ be the result of applying the orthogonal projection
$\Hom_{\CC}\left(\calH_{0},\calH\right)\longrightarrow
\Hom_{\CC\Gamma}\left(\calH_{0},\calH\right)$
to the given inclusion map 
$f:\calH_0\rightarrow\calH$. Then, $h$
is the morphism of representations which is closest to $f$ under
$\|\cdot\|_{\F}$. 
The $\left\{ E_{b_{0},b}\right\} _{\left(b_{0},b\right)\in\calB_{0}\times\calB}$-coefficients of $h$,
which are constant on each $\Gamma$-orbit of $\calB_{0}\times\calB$,
are obtained by taking the average of the coefficients of $f$ in each
$\Gamma$-orbit separately. Write $f=\sum_{\left(b_{0},b\right)\in\calB_{0}\times\calB}\lambda_{b_{0},b}\cdot E_{b_{0},b}$
and $h=\sum_{\left(b_{0},b\right)\in\calB_{0}\times\calB}\mu_{b_{0},b}\cdot E_{b_{0},b}$.
Then, for each $\Gamma$-orbit $\calO$ of $\calB_{0}\times\calB$, the
map $\lambda:\calO\rightarrow\CC$, defined by $\lambda\left(b_{0},b\right)=\lambda_{b_{0},b}$,
has, in its image, $c\left(\calO\right)$ $1$-s and $\left(\left|\calO\right|-c\left(\calO\right)\right)$
$0$-s, while the map $\mu:\calO\rightarrow\CC$, defined by $\mu\left(b_{0},b\right)=\mu_{b_{0},b}$,
is constant, mapping all elements to $\frac{c\left(\calO\right)}{\left|\calO\right|}$.
So, writing $\sum_{\calO}$ for a sum that runs over all $\Gamma$-orbits
$\calO$ of $\calB_0\times\calB$, we deduce that
\begin{align*}
\|f-h\|_{\F}^{2} & =\sum_{\calO}\left(c\left(\calO\right)\cdot\left(1-\frac{c\left(\calO\right)}{\left|\calO\right|}\right)^{2}+\left(\left|\calO\right|-c\left(\calO\right)\right)\cdot\left(0-\frac{c\left(\calO\right)}{\left|\calO\right|}\right)^{2}\right)\\
% & =\sum_{o}\left(c\left(o\right)\cdot\left(1-2\cdot\frac{c\left(o\right)}{\left|o\right|}\right)+\left|o\right|\cdot\left(\frac{c\left(o\right)}{\left|o\right|}\right)^{2}\right)\\
% & =\sum_{o}\left(c\left(o\right)-\frac{c\left(o\right)^{2}}{\left|o\right|}\right)\\
 & =\sum_{\calO}c\left(\calO\right)-\sum_{\calO}\frac{c\left(\calO\right)^{2}}{\left|\calO\right|}
 =\left|\calB_0\right|-\sum_{\calO}\frac{c\left(\calO\right)^{2}}{\left|\calO\right|}\\
\end{align*}
But, using Inequality (\ref{eq:orbits-are-large}), we deduce that
\begin{align*}
\sum_{\calO}\frac{c\left(\calO\right)^{2}}{\left|\calO\right|}=&
\sum_{i=1}^{k}\sum_{\calO\subset \calO^{\calB_0}_i\times\calB}\frac{c\left(\calO\right)^{2}}{\left|\calO\right|} \leq
\sum_{i=1}^{k}\sum_{\calO\subset \calO^{\calB_0}_i\times\calB}\frac{c\left(\calO\right)^{2}}{2\cdot\left|\calO^{\calB_{0}}_i\right|}\\ \leq &
\frac{1}{2}\cdot\sum_{i=1}^{k}\frac{1}{\left|\calO^{\calB_{0}}_i\right|}
\cdot\left(
\sum_{\calO\subset \calO^{\calB_0}_i\times\calB}c\left(\calO\right)
\right)^2
\\ = &
\frac{1}{2}\cdot\sum_{i=1}^{k}\left|\calO^{\calB_{0}}_i\right|
 = 
\frac{1}{2}\cdot\left|\calB_0\right|
\\
\end{align*}
Thus,
\begin{align*}
\|f-h\|_{\F}^{2} \geq & \frac{1}{2}\cdot\left|\calB_{0}\right|
 = \frac{1}{2}\cdot\|f\|_{\F}^{2}\,\,,
\end{align*}
and so, taking square roots finishes the proof.
\end{proof}
We recall the definition of Kazhdan's Property $\T$ (see Section
1.1 of \cite{bhv}). Let $Q\subset\Gamma$ and $\kappa>0$. Recall
that for a unitary representation $\left(\calH,\rho\right)$ of $\Gamma$
and a nonzero vector $v\in\calH$, we say that $v$ is \emph{$\left(Q,\kappa\right)$-invariant}
if $\sup_{x\in Q}\|\rho\left(x\right)\cdot v-v\|<\kappa\cdot\|v\|$. We
say that $\left(Q,\kappa\right)$ is a \emph{Kazhdan pair} for $\Gamma$
if every unitary representation $\left(\calH,\rho\right)$ of $\Gamma$
satisfies:
\begin{align}
\text{} & \text{if \ensuremath{\calH} contains a \ensuremath{\left(Q,\kappa\right)}-invariant vector,}\nonumber \\
 & \text{then it also contains a \ensuremath{\Gamma}-invariant nonzero vector.}\label{eq:property-t}
\end{align}
We say that the group $\Gamma$ has \emph{Kazhdan's Property $\T$}
if it has a Kazhdan pair $\left(Q,\kappa\right)$ for which $Q$ is
finite (and $\kappa>0$). Every discrete group with Property $\T$ is finitely-generated \cites{kaz67,bhv}.
If $\Gamma$ has Property $\T$, then for every finite generating set $Q$
of $\Gamma$, there is $\kappa>0$ for which $\left(Q,\kappa\right)$
is a Kazhdan pair for $\Gamma$, and we call such $\kappa$ a \emph{Kazhdan
constant} for $\left(\Gamma,Q\right)$.
\begin{lem}
\label{lem:under-t-almost-morphism-is-close-to-morphism}Assume that
$\Gamma$ has Property $\T$ with Kazhdan constant $\kappa>0$ for
$\left(\Gamma,S^\pm\right)$. Let $\left(\calH_{1},\alpha\right)$ and
$\left(\calH_{2},\beta\right)$ be finite-dimensional unitary representations
of $\Gamma$. Let $\epsilon>0$, and let $f:\calH_{1}\rightarrow\calH_{2}$
be a nonzero linear map, such that for each $s\in S$,
\[
\|\alpha\left(s^{-1}\right)\circ f\circ\beta\left(s\right)-f\|<\epsilon\cdot\|f\|
\]
Then, there is a morphism $h:\calH_{1}\rightarrow\calH_{2}$ of $\Gamma$-representations,
such that
\[
\|f-h\|<\frac{\epsilon}{\kappa}\cdot\|f\|
\]
\end{lem}

\begin{proof}
The map $f$ is an $\left(S,\epsilon\right)$-invariant vector in
the representation $\Hom_{\CC}\left(\calH_{1},\calH_{2}\right)$ of
$\Gamma$. So, there is a $\Gamma$-invariant linear map $h\in\Hom_{\CC}\left(\calH_{1},\calH_{2}\right)$
such that $\|f-h\|<\frac{\epsilon}{\kappa}\cdot\|f\|$ (see Remark
1.1.10 of \cite{bhv}). The invariance of $h$ is equivalent to $h$
being a morphism of $\Gamma$-representations.
\end{proof}
We are now ready to prove the main theorem:
\begin{proof}
[Proof of Theorem \ref{thm:intro-stable-fi-subgroups}]
Before we begin,
note that for each $n\in\NN$, the group $\Gamma$ has only finitely many finite-index
subgroups of index $n$ because $\Gamma$ is finitely-generated. Since,
in addition, $\Gamma$ has Property $\T$, it has only finitely many irreducible
unitary representations of any given dimension $n\in\NN$ (up to isomorphism).
For the last assertion, see Theorem 2.6 of \cite{wang}, or Corollary
3 of \cite{wasserman} for a more quantitative proof, or Proposition
IV of \cite{hrv} for an explicit upper bound on the number of representations.
In any case, it can be proved by a simple application of Lemma \ref{lem:under-t-almost-morphism-is-close-to-morphism},
together with the compactness of $\U\left(n\right)^{\left|S\right|}$
and Schur's Lemma.

Let $\kappa>0$ be a Kazhdan constant for $\Gamma$ with respect to
$S^\pm$. First, assume that $\Gamma$ is $\calP$-stable. Assume, for the
sake of contradiction, that $\Gamma$ has infinitely many finite-index
subgroups, and let $\left\{ \Lambda_{n}\right\} _{n=1}^{\infty}$
be a sequence of such subgroups for which $\left[\Gamma:\Lambda_{n}\right]\rightarrow\infty$.
Fix $n\in\NN$. Denote $\calB_{n}=\Gamma/\Lambda_{n}=\left\{ x_{1},\dotsc,x_{k}\right\} $,
where $k=\left[\Gamma:\Lambda_{n}\right]$. Write $\alpha_{n}:\Gamma\rightarrow\U$$\left(\CC\left[\calB_{n}\right]\right)$
for the permutation representation produced by the action of $\Gamma$
on $\calB_{n}$ by multiplication from the left. Write $\calB_{n}^{0}=\left\{ x_{1},\dotsc,x_{k-1}\right\} $,
and let $f_n:\CC\left[\calB_{n}^{0}\right]\rightarrow\CC\left[\calB_{n}\right]$
be the inclusion map. By Proposition \ref{prop:restriction-is-almost-solution-with-almost-morphism}(ii),
there is a function $\rho_{n}:S\rightarrow\U\left(\CC\left[\calB_{n}^{0}\right]\right)$,
such that:
\begin{align}
\forall s\in S & \,\,\,\,\text{\ensuremath{\rho_{n}\left(s\right)} permutes \ensuremath{\calB_{n}^{0}}}\label{eq:rho_n-permutes}\\
\forall w\in\Ker\left(\pi\right) & \,\,\,\,\|\rho_{n}\left(w\right)-I\|_{\F}\leq3^{\left|w\right|}\label{eq:rho_n-Frobenius-small}\\
\forall s\in S & \,\,\,\,\|\alpha_{n}\left(s^{-1}\right)\circ f_{n}\circ\rho_{n}\left(s\right)-f_{n}\|_{\F}\leq2\label{eq:almost-equivariance-1}
\end{align}
Inequality (\ref{eq:rho_n-Frobenius-small}) is equivalent to
\begin{align}
\forall w\in\Ker\left(\pi\right) & \,\,\,\,\|\rho_{n}\left(w\right)-I\|_{\HS}\leq\frac{3^{\left|w\right|}}{\left|\calB_{n}^{0}\right|^{1/2}}\,\,.\label{eq:rho_n-HS-small}
\end{align}

From (\ref{eq:rho_n-permutes}) and (\ref{eq:rho_n-HS-small}), we
see that $\left(\rho_{n}\right)_{n=1}^{\infty}$ is a $\calP$-stability-challenge
for $\Gamma$. Since $\Gamma$ is $\calP$-stable, there is a solution
$\left(\tilde{\rho}_{n}\right)_{n=1}^{\infty}$ for $\left(\rho_{n}\right)_{n=1}^{\infty}$,
where $\tilde{\rho}_n:S\rightarrow \U\left(\CC\left[\calB^{0}_{n}\right]\right)$
and $\tilde{\rho}_n\left(s\right)$ is a permutation matrix for each $s\in S$.
We may extend each $\tilde{\rho}_{n}:S\rightarrow\U\left(\CC\left[\calB^{0}_{n}\right]\right)$
to a representation $\tilde{\rho}_{n}:\Gamma\rightarrow\U\left(\CC\left[\calB^{0}_{n}\right]\right)$.
Now, for each $s\in S$,
\begin{align*}
\frac{1}{\|f_{n}\|_{\F}}\cdot\|\alpha_{n}\left(s^{-1}\right)\circ f_{n}\circ\tilde{\rho}_{n}&\left(s\right)-f_{n}\|_{\F}\\
\leq & \left|\calB_{n}^{0}\right|^{-1/2}\cdot\|\alpha_{n}\left(s^{-1}\right)\circ f_{n}\circ\rho_{n}\left(s\right)-f_{n}\|_{\F}\\
 & +\left|\calB_{n}^{0}\right|^{-1/2}\cdot\|\alpha_{n}\left(s^{-1}\right)\circ f_{n}\circ\left(\tilde{\rho}_{n}\left(s\right)-\rho_{n}\left(s\right)\right)\|_{\F}\\
\leq & \left|\calB_{n}^{0}\right|^{-1/2}\cdot2+\|\tilde{\rho}_{n}\left(s\right)-\rho_{n}\left(s\right)\|_{\HS}\,\,\text{,}
\end{align*}
where the last inequality follows from (\ref{eq:almost-equivariance-1})
and the fact that $\alpha_{n}\left(s^{-1}\right)$ and $f_{n}$ are
unitary. Since $\left(\tilde{\rho}_{n}\right)_{n=1}^{\infty}$
is a solution for $\left(\rho_{n}\right)_{n=1}^{\infty}$, we deduce
that
\[
\frac{1}{\|f_{n}\|_{\F}}\cdot\|\alpha_{n}\left(s^{-1}\right)\circ f_{n}\circ\tilde{\rho}_{n}\left(s\right)-f_{n}\|_{\F}\overset{n\rightarrow\infty}{\longrightarrow}0\,\,\text{.}
\]
Thus, by Lemma \ref{lem:under-t-almost-morphism-is-close-to-morphism},
there are morphisms of representations $\left(h_{n}\right)_{n=1}^{\infty}$,
$h_{n}:\CC\left[\calB_{n}^{0}\right]\rightarrow\CC\left[\calB_{n}\right]$,
such that
\[
\frac{1}{\|f_{n}\|_{\F}}\cdot\|f_{n}-h_{n}\|_{\F}\rightarrow0\,\,\text{,}
\]
in contradiction with Proposition \ref{prop:inclusion-map-is-far-from-morphism}(ii).
This finishes the proof under the assumption that $\Gamma$ is $\calP$-stable.

Now, assume that $\Gamma$ is $\HS$-stable rather than $\calP$-stable.
Arguing as above,
using Propositions \ref{prop:restriction-is-almost-solution-with-almost-morphism}(i)
and \ref{prop:inclusion-map-is-far-from-morphism}(i)
instead of Propositions \ref{prop:restriction-is-almost-solution-with-almost-morphism}(ii)
and \ref{prop:inclusion-map-is-far-from-morphism}(ii),
respectively, we deduce that $\Gamma$ has only finitely many irreducible
finite-dimensional representations up to isomorphism.
Assume, for the sake of contradiction,
that $\Gamma$ has infinitely many subgroups of finite-index, and
let $\left\{ \Lambda_{n}\right\} _{n=1}^{\infty}$ be a sequence of
such subgroups, for which $\left[\Gamma:\Lambda_{n}\right]\rightarrow\infty$.
Write $\Lambda_{0}=\Gamma$. We may assume, without loss of generality,
that the subgroups $\left\{ \Lambda_{n}\right\} _{n=1}^{\infty}$
are normal in $\Gamma$, and that $\Lambda_{n}\subsetneq\Lambda_{n-1}$
for all $n\in\NN$. Fix $n\in\NN$. Take $\gamma_{n}\in\Lambda_{n-1}\setminus\Lambda_{n}$.
The regular representation $\CC\left[\Gamma/\Lambda_{n}\right]$ of
the finite group $\Gamma/\Lambda_{n}$ is faithful, and it decomposes
as a direct sum of irreducible representations of $\Gamma$. So, for
at least one of these irreducible representations, call it $V_{n}$,
$\gamma_{n}$ does not act on $V_{n}$ as the identity. But $\gamma_{n}\in\Lambda_{n-1}$,
and so it acts as the identity on $V_{i}$ for each $1\leq i<n$.
Therefore, we produced a sequence $\left\{ V_{n}\right\} _{n=1}^{\infty}$
of pairwise non-isomorphic finite-dimensional irreducible representations
of $\Gamma$, a contradiction.
\end{proof}

\section{\label{sec:AutFn-MCG}The non-$\calP$-stability of $\Aut\left(\FF_n\right)$, $\Out\left(\FF_n\right)$ and $\MCG\left(g\right)$}
In this section, we focus on $\calP$-stability, and write $d_{n}$
for $d_{n}^{\Hamming}.$

\begin{thm}\label{thm:aut-mcg}
~
\renewcommand{\labelenumi}{\roman{enumi})}
\begin{enumerate}
	\item For $n\geq3$, both $\Aut\left(\FF_n\right)$ and $\Out\left(\FF_n\right)$ are not $\calP$-stable.
	\item For $g\geq3$, the mapping class group $\MCG\left(g\right)$
	of a closed orientable surface $\Sigma_g$ of genus $g$
	is not $\calP$-stable.
\end{enumerate}
\end{thm}
The proof uses the following proposition, which is a special case of Proposition A.3 of \cite{BM}.
\begin{prop}\label{prop:P-stability-and-quotients}
	Assume that $\Gamma$ is finitely-presented and $\calP$-stable,
	and let $N$ be a normal subgroup of $\Gamma$. If $N$ is a finitely-generated
	group, then $\Gamma/N$ is $\calP$-stable.
\end{prop}
\begin{rem}
	Proposition \ref{prop:P-stability-and-quotients} is true also in a stronger
	form, where $\Gamma$ is only assumed to be finitely-generated rather than finitely-presented \cite{BLtesting}. We do not need this stronger form here.
\end{rem}
\begin{proof}[Proof of Theorem \ref{thm:aut-mcg}]
Let $n\geq3$ and $g\geq3$. The action of $\Aut\left(\FF_n\right)$ on the abelianization of $\FF_{n}$ and the action of $\MCG\left(g\right)$ on
$H^1\left(\Sigma_g\right)$ give rise to epimorphisms:

\begin{align*}
\MCG(g)& \longrightarrow \Sp_{2g}\left(\ZZ\right) \\
\Aut(\FF_n)& \longrightarrow \GL_n\left(\ZZ\right) \\
\Out(\FF_n)& \longrightarrow \GL_n\left(\ZZ\right) \\
\end{align*}
The groups $\GL_n\left(\ZZ\right)$ and $\Sp_{2g}\left(\ZZ\right)$ are infinite, residually-finite
(hence sofic) and have Property $\T$. Hence, by Theorem \ref{thm:intro-stable-fi-subgroups}, they are not $\calP$-stable.
Moreover, the kernels of all three epimorphisms are the Torelli groups, which are known to be finitely-generated (as $g\geq 3$) \cites{johnson, magnus} (cf. \cite{BKM}).
Therefore, $\Aut(\FF_n)$, $\Out(\FF_n)$ and $\MCG\left(g\right)$ are not $\calP$-stable by virtue of Proposition \ref{prop:P-stability-and-quotients}.
\end{proof}

\section{\label{sec:remarks}Remarks and suggestions for further research}

\subsection{Stability of hyperbolic groups}

It is clear that free groups are both $\calP$-stable and $\HS$-stable.
On the other hand, lattices in the rank one simple Lie groups $\Sp\left(n,1\right)$
($n\geq2$) have Property $\T$ (see \cite{kos75} or \cite{bhv}),
and so they are neither $\calP$-stable nor $\HS$-stable by Theorem
\ref{thm:intro-sofic-hyperlinear}. However both free groups and the
cocompact lattices among the aforementioned lattices are hyperbolic
\cite{gro87}. So, hyperbolicity by itself does not suffice to determine
whether a group is stable. An interesting question is whether surface
groups of genus $g\geq2$ are $\calP$-stable or $\HS$-stable.

Proposition \ref{prop:P-stability-and-quotients} provides an additional method,
which does not have to use Property $\T$, to construct
hyperbolic groups which are not $\calP$-stable.
This can be done through the Rips construction \cite{rips82}: Let $\Lambda$ be a finitely-presented group which is not
$\calP$-stable. For example, take the Baumslag-Solitar group $\Lambda=\BS\left(2,3\right)$, which is not residually-finite \cite{BS}, but is free-by-solvable \cite{Kroph}, hence residually-solvable, and thus sofic.
Then, the Rips construction provides a hyperbolic group $\Omega$ which admits $\Lambda$ as a quotient with
a finitely-generated kernel. Proposition \ref{prop:P-stability-and-quotients} implies that $\Omega$
is not $\calP$-stable.
In Section \ref{subsec:flexi-stab}, we define a more flexible notion
of $\calP$-stability. We remark that if $\Lambda$ is sofic and not residually-finite, then it is not
$\calP$-stable even in the flexible sense. Hence, the same
is true for the hyperbolic group $\Omega$ (as the proof of Proposition \ref{prop:P-stability-and-quotients} works for flexible
$\calP$-stability as well).

\subsection{Property $\protect\ptau$ and Property $\protect\TFD$}

The arguments presented in Section \ref{sec:proofs} do not require
the full strength of Property $\T$ in the sense that they only go
through \emph{finite-dimensional} unitary representations of $\Gamma$.
Focusing on $\calP$-stability (rather than $\HS$-stability), even more
is true: only finite-dimensional unitary representations that \emph{factor
through finite quotients} of $\Gamma$ are relevant. Recall that
the finitely-generated group $\Gamma$ has Property $\ptau$ if it has a pair $\left(Q,\kappa\right)$,
$\left|Q\right|<\infty$, $\kappa>0$, such that Condition (\ref{eq:property-t})
from the definition of Property $\T$ holds for all finite-dimensional
representations of $\Gamma$ that factor through finite quotients,
and it has Property $\TFD$ (see \cite{lz}) if the same holds for
\emph{all} finite-dimensional representations of $\Gamma$. We get
the following more general result:
\begin{thm}
\textbf{\label{thm:main-tau}}\renewcommand{\labelenumi}{\roman{enumi})}
\begin{enumerate}
\item If $\Gamma$ has Property $\ptau$ and is $\calP$-stable, then $\Gamma$
has only finitely many finite-index subgroups. Hence, a sofic group
with Property $\ptau$ is not $\calP$-stable, unless it is finite.
\item If $\Gamma$ has Property $\TFD$ and is $\HS$-stable, then $\Gamma$
has only finitely many finite-index subgroups. Hence, a hyperlinear
group with Property $\TFD$ is not $\HS$-stable, unless it is finite.
\end{enumerate}
\end{thm}

\textbf{Warning:} The weaker notion of Property $\ptau$ \emph{with
respect to a family of finite-index subgroups} $\left\{ N_{i}\right\} _{i=1}^{\infty}$
does \emph{not} suffice to deduce the conclusion of Theorem \ref{thm:main-tau}(i),
even if the family is separating (i.e. $\cap N_{i}=\left\{ 1\right\} $).
For example, the group $\Gamma=\langle\left(\begin{array}{cc}
1 & 2\\
0 & 1
\end{array}\right),\left(\begin{array}{cc}
1 & 0\\
2 & 1
\end{array}\right)\rangle$ is free, so it is clearly $\calP$-stable, and has Property $\ptau$
with respect to the family of \emph{congruence} subgroups $\left\{\Gamma\cap\Ker\left(\SL_{2}\left(\ZZ\right)\rightarrow\SL_{2}\left(\ZZ/m\ZZ\right)\right)\right\}_{m=1}^{\infty}$
(it has the so called \emph{Selberg property} \cite{LW}).

Note that it is easy to see that a free product of stable groups is
stable (for all versions of stability). An interesting corollary of
Theorem \ref{thm:main-tau}(i) is that a free product of two $\calP$-stable
groups, amalgamated along a finite-index subgroup, is not necessarily
$\calP$-stable. Indeed, for $p$ an odd prime, look at
\[
\Gamma\left(2\right)=\Ker\left(\SL_{2}\left(\ZZ\left[\frac{1}{p}\right]\right)\rightarrow\SL_{2}\left(\ZZ\left[\frac{1}{p}\right]/2\ZZ\left[\frac{1}{p}\right]\right)\right)\,\,\text{.}
\]
This is an amalgamated product of two free groups along a finite index
subgroup (see \cite{serre}, Chapter II, Section 1.4, Corollary 2), and,
as with the example of $\SL_2\left(\ZZ\right)$ above, it has the Selberg property \cite{LW}.
However, unlike $\SL_2\left(\ZZ\right)$, the group $\SL_{2}\left(\ZZ\left[\frac{1}{p}\right]\right)$
satisfies the congruence subgroup property \cite{serre70}, and so from the Selberg property
we deduce that it has Property $\ptau$, and so the same is true for
$\Gamma\left(2\right)$,
hence the latter is not $\calP$-stable.

\subsection{Relative Property $\protect\T$}

Recall that the group $\Gamma$, generated by the finite set $S$, has \emph{relative Property $\T$}
with respect to a subgroup $N\leq\Gamma$ if there is $\kappa>0$,
such that every unitary representation $\left(\calH,\rho\right)$
of $\Gamma$ that has an $\left(S,\kappa\right)$-invariant vector
$v\in\calH$, also has an $N$-invariant nonzero vector. If $\Gamma$
has relative Property $\T$ with respect to a subgroup $N\leq\Gamma$,
rather than Property $\T$, we may deduce a weak form of Lemma \ref{lem:under-t-almost-morphism-is-close-to-morphism},
where the constructed morphism $h$ is merely a morphism of $N$-representations.
Using this variant of the lemma, and arguing as in the proof of Theorem
\ref{thm:intro-stable-fi-subgroups}, we deduce the following:
\begin{thm}
\label{thm:relative-t}Assume that $\Gamma$ is $\calP$-stable and has
relative Property $\T$ with respect to a subgroup $N\leq\Gamma$. Then, the
collection 
\[
\left\{ L\leq\Gamma\mid\left[\Gamma:L\right]<\infty\,\,\text{and}\,\,\Gamma=NL\right\} 
\]
is finite.
\end{thm}

We exhibit an application of Theorem \ref{thm:relative-t}. The group
$\SL_{2}\left(\ZZ\right)$ acts on $\ZZ^{2}$ by matrix multiplication,
giving rise to a semi-direct product $\ZZ^{2}\rtimes\SL_{2}\left(\ZZ\right)$.
It is well-known that this semi-direct product has relative Property
$\T$ with respect to the subgroup $\ZZ^{2}\rtimes\left\{ 1\right\} $.
So, the infinite collection $\left\{ \left(n\ZZ^{2}\right)\rtimes\SL_{2}\left(\ZZ\right)\right\} _{n=1}^{\infty}$,
of finite-index subgroups, exhibits the non-$\calP$-stability of $\ZZ^{2}\rtimes\SL_{2}\left(\ZZ\right)$.
More interestingly, letting $H$ be the finite-index
subgroup of $\SL_{2}\left(\ZZ\right)$ generated by $\left(\begin{array}{cc}
1 & 2\\
0 & 1
\end{array}\right)$ and $\left(\begin{array}{cc}
1 & 0\\
2 & 1
\end{array}\right)$, we may deduce in the same manner that $\ZZ^{2}\rtimes H$ is not
$\calP$-stable as well. Note that since $\ZZ^{2}$ is abelian and $H$
is free, we know that both are stable \cite{arzhantseva-paunescu}! We conclude:
\begin{cor}
A semidirect product of finitely-presented
$\calP$-stable groups is not necessarily $\calP$-stable.
\end{cor}

\subsection{\label{subsec:flexi-stab}A flexible variant of $\protect\calP$-stability}

Finally, let us make a remark and a suggestion for further research.
Our proof of non-$\calP$-stability of groups with Property $\T$
starts with a true action
of $\Gamma$ on a set $X$ of $n$ points, which is then deformed a bit
into an almost action on a set of $n-1$ points. For $\Gamma$ to be $\calP$-stable,
this almost action must be close to an actual action on $n-1$ points.
We proved that it is never the case if $\Gamma$ has Property $\T$ and the action
$\Gamma\curvearrowright X$ is transitive. However, the action on $n-1$ points is clearly close to a
true action on a set of $n$ points since we started with such an
action.

One may suggest a notion of ``flexible $\calP$-stability'', which
requires that every almost action can be corrected to an action
by allowing to add points to the set before correcting
it.
One then needs to decide how to measure the distance between permutations of different sizes.
For $\sigma\in\Sym\left(n\right)$ and $\tau\in\Sym\left(m\right)$, $n\leq m$,
we generalize $d_n^{\Hamming}$ by defining
\[
d_n^{\flex}\left(\sigma,\tau\right)=\frac{1}{n}\cdot
\left(\left|\left\{ x\in\left[n\right]\mid\sigma\left(x\right)\neq\tau\left(x\right)\right\} \right|+\left(m-n\right)\right)\,\,\text{.}
\]
For a $\calP$-stability-challenge
$\left(f_k:S\rightarrow\Sym\left(n_k\right)\right)_{k=1}^\infty$,
we define a "flexible solution" to be a sequence of homomorphisms
$\left(g_k:\Gamma\rightarrow\Sym\left(t_k\right)\right)_{k=1}^\infty$, $t_k\geq n_k$,
for which $d_{n_k}^{\flex}\left(f_k\left(s\right),g_k\left(s\right)\right)\overset{k\rightarrow\infty}{\longrightarrow}0$
for each $s\in S$.
We say that $\Gamma$ is $\calP$-\emph{flexibly}-stable if every
$\calP$-stability-challenge for $\Gamma$ has a flexible solution.

One may consider an even more flexible notion of $\calP$-stability
by replacing $d_n^{\flex}$ with
\[
d_n^{\veryflex}\left(\sigma,\tau\right)=\frac{1}{n}\cdot
\left|\left\{ x\in\left[n\right]\mid\sigma\left(x\right)\neq\tau\left(x\right)\right\} \right|
\,\,\text{.}
\]

The flexible notions of stability suggest a path towards finding
a non-sofic group: a non-residually-finite group which
is $\calP$-flexibly-stable,
or even just $\calP$-very-flexibly-stable,
is non-sofic.
In fact, there is a possibly more accessible path which requires less.
A \emph{sofic approximation} for $\Gamma$ is a sequence
$\left(f_k:S\rightarrow\Sym\left(n_k\right)\right)_{k=1}^\infty$,
$n_k\in\NN$, such that
\begin{align*}
\forall w\in\Ker\left(\pi \right) \qquad & d_{n_{k}}\left(f_{k}\left(w\right),{\bf 1}_{G_{n_{k}}}\right)\overset{k\rightarrow\infty}{\longrightarrow}0 \\
\forall w\notin\Ker\left(\pi \right) \qquad & d_{n_{k}}\left(f_{k}\left(w\right),{\bf 1}_{G_{n_{k}}}\right)\overset{k\rightarrow\infty}{\longrightarrow}1\,\,\text{.}
\end{align*}
Note that every sofic approximation is a $\calP$-stability-challenge.
Now, if a finitely-generated group $\Gamma$ is not residually-finite, then $\Gamma$ must be non-sofic if it satisfies the following condition:
For every sofic approximation 
$\left(f_k:S\rightarrow\Sym\left(n_k\right)\right)_{k=1}^\infty$
for $\Gamma$, there is a sequence of homomorphisms
$\left(g_k:\Gamma\rightarrow\Sym\left(t_k\right)\right)_{k=1}^\infty$,
$t_k\geq n_k$, such that
\[
\liminf_{k\rightarrow\infty}
\sum_{s\in S}d
^{\veryflex}_{n_{k}}\left(f_{k}\left(s\right),g_{k}\left(s\right)\right)
=0\,\,\text{.}
\]
In fact, it is enough to require that the above $\liminf$ is smaller than some small enough
constant depending on $\left|S\right|$ and on the length of the shortest word exhibiting the non-residual-finiteness
of $\Gamma$.
This strategy generalizes the observation \cite{glebsky-rivera} that
a non-residually-finite $\calP$-stable group is non-sofic,
and its strengthening \cite{arzhantseva-paunescu} which says
the same with respect to the notion of weak-stability,
i.e., stability with respect to sofic approximations only, rather
than general $\calP$-stability-challenges.

One may hope that this path would make it easier to find a non-sofic group.
Note that it is still an open problem whether surface groups are
$\calP$-stable, but in response to an earlier version of the present paper,
Lazarovich, Levit and Minsky proved that
they are $\calP$-flexibly-stable \cite{lazarovich-levit-minsky}.

\appendix
\section{Non-stability through almost-irreducible actions}
In this appendix, we give an alternative proof of Theorem
\ref{thm:intro-sofic-hyperlinear} under the assumption that $\Gamma$
has an infinite linear quotient.
In fact, we prove a stronger statement:
\begin{thm}\label{thm:appendix-linear-group-not-stable}
	Let $\Gamma$ be a finitely generated group with Property $\T$. Assume that $\Gamma$ has an infinite linear
	quotient. Then $\Gamma$ has a $\calP$-stability-challenge which has no solution, even
	when thought of as an $\HS$-stability-challenge.
	In particular, $\Gamma$ is neither $\calP$-stable nor $\HS$-stable.
\end{thm}

The main ingredient in the proof of Theorem \ref{thm:appendix-linear-group-not-stable} is the following result, proved below:
\begin{thm}\label{thm:appendix-good-representations}
	Let $\Gamma$ be a finitely generated group having
	a non-virtually-solvable infinite linear quotient.
	Then it has an infinite sequence
	of finite permutational representations $\rho_{n}:\Gamma\rightarrow\Sym\left(X_{n}\right)$
	such that the corresponding unitary representations $\bar{\rho}_{n}:\Gamma\rightarrow\U\left(L^{2}\left(X_{n}\right)\right)$
	are almost irreducible in the following sense: $\bar{\rho}_{n}$ has
	an irreducible subrepresentation $\chi_{n}$ such that $\frac{\dim\chi_{n}}{\dim\bar{\rho}_{n}}\overset{n\rightarrow\infty}{\longrightarrow}1$.
\end{thm}
We begin by explaining how Theorem \ref{thm:appendix-linear-group-not-stable} follows from Theorem \ref{thm:appendix-good-representations}.
Let $\Gamma$ be a finitely generated group that has Property $\T$ and an infinite linear quotient.
This linear quotient is not virtually solvable by virtue of Property $\T$, hence Theorem \ref{thm:appendix-good-representations} is applicable to $\Gamma$.
As usual, we think of $\Gamma$ as a quotient of a free group $\FF$, generated by a finite set $S$.
Take $\rho_{n}:\Gamma\rightarrow\Sym\left(X_{n}\right)$ and $\chi_n$ as in Theorem \ref{thm:appendix-good-representations}.
Choose a subset $Y_n$ of $X_n$ whose cardinality is $\left|Y_n\right|=\dim\chi_n$.
Choose one action $\FF\curvearrowright Y_n$ satisfying the following:
for every $s\in S$ and $x\in X_n$, if $s$ takes $x$ to $x'$ under the action
$\FF\curvearrowright X_n$ and both $x$ and $x'$ belong to $Y_n$,
then $s$ takes $x$ to $x'$ under the action $\FF\curvearrowright Y_n$ as well.
Let $\bar{\varphi}_n:S\rightarrow\U\left(L^2\left(Y_n\right)\right)$ be the unitary representation
corresponding to the action $\FF\curvearrowright Y_n$.
The sequence $\left(\bar{\varphi}_n\right)_{n=1}^{\infty}$ is an $\HS$-stability-challenge for $\Gamma$, and can also be thought of as a $\calP$-stability-challenge.
Assume, for the sake of contradition, that this $\HS$-stability-challenge has
a solution.
Then, for large enough $n$, the map $\bar{\varphi}_n$ is very close to a map
$\tilde{\varphi}_n:S\rightarrow\U\left(L^2\left(Y_n\right)\right)$ which extends to a homomorphism
$\tilde{\varphi}_n:\Gamma\rightarrow\U\left(L^2\left(Y_n\right)\right)$.
Composing the canonical inclusion $L^2\left(Y_n\right)\rightarrow L^2\left(X_n\right)$ with a projection of $\Gamma$-representations
$L^2\left(X_n\right)\rightarrow\chi_n$ produces a linear map
$f:L^2\left(Y_n\right)\rightarrow \chi_n$.
The map $f$ is an almost invariant non-zero vector in the $\Gamma$-representation
$\Hom_{\CC}\left(\left(L^2\left(Y_n\right),\tilde{\varphi}_n\right),\chi_n\right)$.
Hence, since $\Gamma$ has Property $\T$, there is a non-zero morphism of
$\Gamma$-representations
$\left(L^2\left(Y_n\right),\tilde{\varphi_n}\right)\rightarrow\chi_n$.
This contradicts Schur's Lemma since $\left|Y_n\right|=\dim\chi_n$ and so every irreducible subrepresentation
of $\left(L^2\left(Y_n\right),\tilde{\varphi_n}\right)$ is of dimension smaller or equal to $\left|Y_n\right|-1=\dim\chi_n-1$.

\vspace{2mm}
Our proof of Theorem \ref{thm:appendix-good-representations} makes use of the following claim. Thanks are due to Bob Guralnick for useful discussions of it.
\begin{prop}\label{prop:appendix-chevalley}
Let $\utilde{G}$ be a semisimple Chevalley group scheme and $\utilde{B}$ its Borel subgroup.
Write $G=\utilde{G}\left(\FF_{q}\right)$ and $B=\utilde{B}\left(\FF_{q}\right)$, where $\FF_q$ is a
finite field of order $q$. Then, there is an irreducible $G$-subrepresentation
$V$ of $L^{2}\left(G/B\right)$
such that $\lim_{q\rightarrow\infty}\frac{\dim V}{\dim L^{2}\left(G/B\right)}=1$.
\end{prop}
\begin{proof}
Let $p=\chara \FF_{q}$.
Assume first that $\utilde{G}$ is a simple Chevalley group.
Denote $r=\rank\left(\utilde{G}\right)$ and write $R$ for the number of positive roots of $G$.
The Steinberg representation $\St$ of $G$ is a subrepresentation of $L^2\left(G/B\right)$
(see Section 1.3 of \cite{humphreys87}).
By Theorem 6.4.7(ii) of \cite{carter}, $\dim\St$ is equal to the cardinality of the $p$-Sylow subgroup $S_{p}$
of $G$. By Lemma 54 and its corollary in \cite{steinberg-notes}, $\left|S_{p}\right|=q^R$,
$\left|B\right|=q^R\cdot (q-1)^r$ and 
\begin{align}\label{eq:size-of-G}
\left|G\right|&=q^{R}\left(q-1\right)^{r}\cdot\sum_{w\in W}q^{N\left(w\right)}
\end{align}
where the sum runs over the elements
of the Weyl group $W$ associated with $\utilde{G}$, and $N\left(w\right)$
is the number of positive roots taken by $w$ to negative roots.
Exactly one element $w_{0}$ of $W$, the longest element, takes
all positive roots to the negative ones (\cite{humphreys}, Section 1.8). Hence, the sum in
Equation (\ref{eq:size-of-G}) is equal to
\[
q^{R}+\sum_{w_{0}\neq w\in W}q^{N\left(w\right)},\qquad N\left(w\right)<R\,\,\text{.}
\]
This sum is the index of $B$ in $G$,
which is equal to $\dim L^{2}\left(G/B\right)$.
Hence, $\lim_{q\rightarrow\infty}\frac{\dim \St}{\dim L^{2}\left(G/B\right)}=1$.

Now, consider the general case where $\utilde{G}$ is a semisimple Chevalley group.
We can assume that $\utilde{G}=\prod_{i=1}^{s}\utilde{G}_{i}$ with $\utilde{G}_{i}$
simple. Hence, $\utilde{G}\left(\FF_{q}\right)/\utilde{B}\left(\FF_{q}\right)=\prod_{i=1}^{s}\utilde{G}_{i}\left(\FF_{q}\right)/\utilde{B}_{i}\left(\FF_{q}\right)$
and the tensor product of the Steinberg representations of the $\utilde{G}_{i}\left(\FF_{q}\right)$
is an irreducible subrepresentation of $L^{2}\left(G/B\right)$ of
almost full dimension by what we proved in the case of a simple Chevalley group.
\end{proof}
In the notation of the claim, the $G$-set $G/B$ can be identified
with the set $\conj_{G}\left(B\right)$ of conjugates of $B$ in $G$
since $B$ is self-normalizing in $G$.

\begin{proof}
[Proof of Theorem \ref{thm:appendix-good-representations}]
Assume W.L.O.G. that $\Gamma$
is an infinite linear group which is not virtually
solvable. Dividing its Zarsiki closure $H$ by the solvable radical,
we can assume that $H$ is non-trivial and semisimple (though maybe
not connected).
Moreover, we can assume (cf. \cite{LV}, Lemma 3.6) that $H^0$, the
identity component of $H$, is simply connected, and $H^0$ is
of the form $\prod_{i=1}^{s}\utilde{G}_i$ where $\utilde{G}_i$
are simple Chevalley groups.

By Section 4 of \cite{LaLu}, we can replace $\Gamma$ with a specialization
of $\Gamma$, and assume that $\Gamma\subset\GL_{n}\left(k\right)$,
where $k$ is a global field, and still the Zariski closure of $\Gamma$ 
is a $k$-algebraic group which is isomorphic to $H$ over the algebraic closure $\bar{k}$ of $k$, and we replace $H$ by this new Zariski closure.
As $\Gamma$ is finitely-generated, it lies
in $H\left(\calO_{S}\right)$ where $\calO$ is the ring of integers
of $k$ and $\calO_{S}$ is its localization
at a finite set $S$ of primes.

We now apply Nori-Weisfeiler strong approximation
(cf. \cite{lubotzky-segal},
Window: Strong approximation for linear groups)
to $\Gamma^{0}=H^{0}\cap\Gamma$
to deduce that for almost every prime ideal $\frakp$ of $\calO_{S}$,
$\Gamma^{0}$ is mapped onto $H^0\left(\calO_{S}/\frakp\right)$.
Moreover, applying the Chebotarev Density Theorem
(as in \cite{lubotzky-nikolov}, Section 4),
we deduce that if $\utilde{G}$ is the Chevalley (split) form of $H^0$, then
for infinitely many $\frakp$ (in fact, a subset of positive density),
$H^0\left(\calO_{S}/\frakp\right)\cong \utilde{G}\left(\calO_{S}/\frakp\right)$.
To summarize, $\Gamma^0$ sujects onto $G=\utilde{G}\left(\FF_q\right)$ for infinitely
many finite fields $\FF_q$. Fixing one such field $\FF_q$ and letting $B$ the the Borel
subgroup of $G$, the action $\Gamma^0\curvearrowright \conj_G\left(B\right)$ extends to an
action $\Gamma\curvearrowright \conj_G\left(B\right)$ since $\Gamma^0$ is a normal
subgroup of $\Gamma$. The desired conclusion now follows from Proposition \ref{prop:appendix-chevalley}, the identification of $G/B$ with $\conj_G\left(B\right)$ and the fact that the largest dimension of a
subrepresentation of $L^2\left(G/B\right)$ does not decrease when extending the
$\Gamma^0$-representation to a $\Gamma$-representation.
\end{proof}

\end{document}